\documentclass[12pt,a4paper]{article}
\usepackage{amsmath}
\usepackage{amssymb}
\usepackage{amsthm}
\usepackage{amsfonts}
\usepackage{latexsym}
\usepackage{verbatim} 
\usepackage{xcolor}
\usepackage[a4paper,width=16cm,top=2cm,bottom=2cm,footskip=1cm]{geometry}

\usepackage[flushmargin]{footmisc}
\usepackage{color}                    
\usepackage{esvect}
\usepackage{mathrsfs}

\theoremstyle{plain}
\newtheorem{thm}{Theorem}[section]

\newtheorem{lem}{Lemma}[section]

\newtheorem{ass}{Assumption}[section]
\theoremstyle{remark}

\theoremstyle{definition}
\newtheorem{defn}{Definition}[section]

\newtheorem{rem}{Remark}[section]

\newcommand{\Complex}{\mathbb C}
\newcommand{\Real}{\mathbb R}

\newcommand{\ddbar}{\overline\partial}
\newcommand{\pr}{\partial}
\newcommand{\ol}{\overline}

\newcommand{\abs}[1]{\left\vert#1\right\vert}
\newcommand{\set}[1]{\left\{#1\right\}}
\newcommand{\To}{\rightarrow}

\title{Commutativity of quantization with conic reduction for torus actions on compact CR manifolds}
\author{Andrea Galasso\footnote{\noindent{\bf Address:} Dipartimento di Matematica e Applicazioni, Universit\`a degli Studi di Milano-Bicocca, Via R.	Cozzi 55, 20125 Milano, Italy; \\ {\bf ORCID iD:} 0000-0002-5792-1674; {\bf e-mail}: andrea.galasso@unimib.it andrea.galasso.91@gmail.com \\  {\bf Scholarship: } Titolare di una borsa per l'estero dell'Istituto Nazionale di Alta Matematica}}

\date{}

\begin{document}
	\maketitle

\begin{abstract} 
	We define conic reductions $X^{\mathrm{red}}_{\nu}$ for torus actions on the boundary $X$ of a strictly pseudo-convex domain and for a given weight $\nu$ labeling a unitary irreducible representation. There is a natural residual circle action on $X^{\mathrm{red}}_{\nu}$. We have two natural decompositions of the corresponding Hardy spaces $H(X)$ and $H(X^{\mathrm{red}}_{\nu})$. The first one is given by the ladder of isotypes $H(X)_{k\nu}$, $k\in\mathbb{Z}$, the second one is given by the $k$-th Fourier components $H(X^{\mathrm{red}}_{\nu})_k$ induced by the residual circle action. The aim of this paper is to prove that they are isomorphic for $k$ sufficiently large. The result is given for spaces of $(0,q)$-forms with $L^2$-coefficient when $X$ is a CR manifold with non-degenerate Levi form.  
\end{abstract}
\tableofcontents
	


\textbf{Data Availability Statements:} Data sharing not applicable to this article as no datasets were generated or analyzed during the current study.

\section{Introduction}

Let $X$ be the boundary of a strictly pseudo-convex domain $D$ in $\mathbb{C}^{n+1}$. Then $(X,\,T^{1,0}X)$ is a contact manifold of dimension $2n+1,\, n\geq 1$, where $T^{1,0}X$ is the sub-bundle of $TX\otimes \mathbb{C}$ defining the CR structure. We denote by $\omega_0\in \mathcal{C}^{\infty}(X,T^*X)$ the contact $1$-form whose kernel is the horizontal bundle $HX\subset TX$, we refer to Section~\ref{sec:Geom} for definitions. Associated with this data we can define the Hardy space $H(X)$, it is the space of boundary values of holomorphic functions in $D$ which lie in $L^2(X)$, the Hilbert space of square integrable functions on $X$. Suppose a contact and CR action of a $t$-dimensional torus $\mathbb{T}$ is given; we denote by $\mu: X\rightarrow \mathfrak{t}^*$ the associated CR moment map. Fix a weight $i\,\nu$ in the lattice $i\,\mathbb{Z}^t\subset \mathfrak{t}^*$, if $0\in \mathfrak{t}^*$ does not lie in the image of the moment map the isotypes
\[H(X)_{k\nu}=\{f\in H(X)\,:\, (e^{i\, \theta}\cdot f)(x) = e^{ik\,\langle \nu,\theta\rangle}f(x),\, \theta \in \mathbb{R}^t  \}, \quad k\in\mathbb{Z}, \]
are finite dimensional.

Suppose that the ray $i\,\mathbb{R}_+\cdot \nu\in \mathfrak{t}^*$ is transversal to $\mu$, then $X_{\nu}:=\mu^{-1}(i\,\mathbb{R}_+\cdot \nu)$ is a sub-manifold of $X$ of codimension $t-1$. There is a well-defined locally free action of $\mathbb{T}^{t-1}_{\nu}:=\exp_{\mathbb{T}}(i\,\ker\nu)$ on $X_{\nu}$, the resulting orbifold $X^{\mathrm{red}}_{\nu}$ is called \textit{conic reduction} of $X$ with respect to the weight $\nu$. Let $\varphi$ be an Euclidean product on $\mathfrak{t}$, we shall also use the symbol $\langle\cdot\,,\,\cdot \rangle$, and denote by ${\lambda}^{\varphi}\in \mathfrak{t}$ be uniquely determined by ${\lambda}=\varphi({\lambda}^{\varphi},\,\cdot)$ and $\lVert {\lambda}\rVert$ the corresponding norm. By abuse of notation we write $\lambda$ for ${\lambda}^{\varphi}$ and we identify $\mathfrak{t}\cong i \mathbb{R}^t$ with its dual. We set
\[\ker \nu= \nu^{\perp}:=\left\{{\lambda}\in \mathfrak{t}\,:\,\langle \nu,\, {\lambda} \rangle=0 \right\}\,.\]

The locus $X_{\nu}$ is $\mathbb{T}$-invariant, we will always assume that the action of $\mathbb{T}$ on $X_{\nu}$ is locally free. After replacing $\mathbb{T}$ with its quotient by a finite subgroup, we may and will assume without loss of generality that the action is generically free. In Section~\ref{sec:Geom}, we show that $X^{\mathrm{red}}_{\nu}$ is CR manifold with positive definite Levi form of dimension $2n-2t+3$. 

Let us define $\mathbb{T}^1_{\nu}:=\exp_{\mathbb{T}}(i\,\nu)$, if $\nu$ is coprime, we have a Lie group isomorphism
\[\kappa_{\nu}\,:\,S^1\rightarrow \mathbb{T}^1_{\nu},\quad e^{i\theta}\mapsto e^{i\theta\nu}\]
between $\mathbb{T}^1_{\nu}:=\exp_{\mathbb{T}}(i\,\nu)$ and the circle $S^1$. Let us denote by $$\overline{\mathbb{T}^1_{\nu}}:=\mathbb{T}/ \mathbb{T}^{t-1}_{\nu}\cong \mathbb{T}^1_{\nu}/ (\mathbb{T}^1_{\nu}\cap \mathbb{T}^{t-1}_{\nu})\,,$$ then the character $\chi_{\nu}\,:\,\mathbb{T}\rightarrow S^1,\,\chi_{\nu}(e^{i\,\theta}):=e^{ik\,\langle \nu,\,\theta\rangle}$, being trivial on $\mathbb{T}^{t-1}_{\nu}$, descends to a character $\chi_{\nu}'\,:\,\overline{\mathbb{T}^1_{\nu}}\rightarrow S^1$ which is a Lie group isomorphism, see \cite[Lemma 10]{Pao2}.
Thus, we have a locally free circle action of $\overline{\mathbb{T}^1_{\nu}}$ on $X^{\mathrm{red}}_{\nu}$, which induces an action on the Hardy space $H(X^{\mathrm{red}}_{\nu})$. Suppose that the action of $\overline{\mathbb{T}^1_{\nu}}$ on $X$ is transversal to the CR structure. We denote by  $H(X^{\mathrm{red}}_{\nu})_k$ the corresponding $k$-th Fourier component and we call the action of $\overline{\mathbb{T}^1_{\nu}}$ on $X^{\mathrm{red}}_{\nu}$ \textit{residual circle action}. The aim of this paper is to prove that $H(X)_{k\nu}$ and $H(X^{\mathrm{red}}_{\nu})_k$ are isomorphic for $k$ sufficiently large.


We prove the aforementioned result in the more general setting of CR manifolds for spaces of $(0,q)$-forms when $k$ is large, more precisely we  consider $(0,q)$ forms with $L^2$ coefficients and the corresponding projector $S^{(q)}$ onto the kernel of the Kohn Laplacian $H^q(X)$. Now, we make more precise the assumptions on the CR manifold $X$ and on the group action.

\begin{ass}
	Let $(X,\, T^{1,0}X)$ be a compact connected orientable CR manifold of dimension $2n+1$, $n\geq 1$, and let $\omega$ be the associated contact $1$-form. The Levi form $L$ is non-degenerate of constant signature $(n_-,\,n_+)$ on $X$. That is, the Levi form has exactly $n_-$ negative and $n_+$ positive eigenvalues at each point of $X$, where $n_-+n_+=n$.
\end{ass}  

Concerning the group action, we always assume 

\begin{ass} \label{as:2}
	The action of $\mathbb{T}$ preserves the contact form $\omega_0$ and the complex structure $J$. That is, $g^*\omega_0=\omega_0$ on $X$ and $g_*J=Jg_*$ on the horizontal bundle $HX$ for every $g\in \mathbb{T}$ where $g^*$ and $g_*$ denote the pull-back map and push-forward map of $\mathbb{T}$, respectively.	
\end{ass}

Let $X^{\mathrm{red}}_{\nu}:=X_{\nu}/\mathbb{T}^{t-1}_{\nu}$, defined in the same way as before, more precisely we shall assume

\begin{ass} \label{as:3}
	The moment map $\mu$ is transverse to the ray $i\,\mathbb{R}_+\cdot \nu\in \mathfrak{t}^*$, the action of $\mathbb{T}$ on $X_{\nu}$ is locally free and for every $x\in X_{\nu}$
	\[\mathrm{val}_x(\nu^{\perp})\cap \mathrm{val}_x(\nu^{\perp})^{\perp_b}=\{0\}  \]
	where $b$ is a bilinear form on $H_xX$ such that 
	\begin{equation}\label{E:biform}
		b(\cdot \,,\, \cdot) = \mathrm{d}\omega_0(\cdot \,,\, J\cdot) 
	\end{equation}
	and it is non-degenerate.	
\end{ass}  

We note that $b(U,\,V)=2\,L(U,\,V)$ for every $U,\,V\in HX$. By Assumptions above, we will show that $X^{\mathrm{red}}_{\nu}$ is a CR manifold with natural CR structure induced by $T^{1,0}X$ of dimension $2n - 2(t-1) + 1$. Let $L_{X^{\mathrm{red}}_{\nu}}$ be the Levi form on $X^{\mathrm{red}}_{\nu}$ induced naturally from the Levi form $L$ on $X$. For a given subspace $\mathfrak{s}$ of $\mathfrak{t}$, we denote $\mathfrak{s}_X$ the subspace of infinitesimal vector fields on $X$. Let us consider 
\[B=\ker{\nu}_X\oplus J\ker{\nu}_X\,.\]
Hence, $b$ has constant signature on $B\times B$, suppose $b$ has $r$ negative eigenvalues on $B\times B$ where $r \leq n_-$ since $L$ and $b$ have the same number of negative eigenvalues on $HX$. Fix $q=n_-$, hence by Lemma \ref{lem:2.1}, $L_{X^{\mathrm{red}}_{\nu}}$ has $q-r$ negative eigenvalues at each point of ${X^{\mathrm{red}}_{\nu}}$. We refer to Section~\ref{sec:Geom} for definitions, we have:

\begin{thm} \label{thm:forms}
	Suppose that $\Box^q_b$ has $L^2$ closed range. Fix a maximal coprime weight $\nu\neq 0$ in the lattice inside $\mathfrak{t}^*$ and assume that the circle action $\overline{\mathbb{T}^1_{\nu}}$ is a transversal CR action. Fix $q=n_-$, under the assumptions above, $X^{\mathrm{red}}_{\nu}$ is a compact CR manifold with non-degenerate Levi form having $q-r$-negative eigenvalues. There is a natural isomorphism of vector spaces $\sigma_k\,:\, H^q(X)_{k\nu}\rightarrow H^{q-r}(X^{\mathrm{red}}_{\nu})_k$ for $k$ sufficiently large.
\end{thm}

For strictly pseudoconvex domain we have $q=n_-=r=0$ and thus we have quantization commutes with reduction for spaces of functions for $k$ large. We give a proof in Section \ref{sec:forms}, which is inspired from \cite{hsiaohuang} (see \cite{hhm} for the full extension of \cite{hsiaohuang}) and it is a consequence of the microlocal properties of the projector $S^{(q)}_{k\nu}$ described in Section \ref{sec:Hardy} and calculus of Fourier integral operators of complex type, see \cite{ms}. Furthermore, we recall that the way to establish the isometry from kernel expansion for $k$ large comes from \cite{mz}.

The conic reductions defined above appear naturally in geometric quantization. In fact, given a Hamiltonian and holomorphic action with moment map $\Phi$ of a compact Lie Group $G$ on a Hodge manifold $(M,\omega)$ with quantizing circle bundle $\pi:X\rightarrow M$, one can always define an infinitesimal action of the Lie algebra $\mathfrak{g}$ on $X$. If it can be integrated to an action of the whole group $G$, then one has a representation of $G$ on $H(X)$. In~\cite{gs2} it was observed that associated to group actions one can define reductions $M_{\mathrm{red}}^{\Theta}$ by pulling back a $G$-invariant and proper sub-manifold $\Theta$ of $\mathfrak{g}^*$ via the moment map $\Phi$. When $\Theta$ is chosen to be a cone through a co-adjoint orbit $C(\mathcal{O}_{\nu})$ one has associated reduction whose Hardy space of its quantization is
$$H(X_{\mathrm{red}}^{C(\mathcal{O}_{\nu})})= \bigoplus_{k\in \mathbb{Z}} H(X_{\mathrm{red}}^{C(\mathcal{O}_{\nu})})_k\,,$$
where $k$ labels an irreducible representation of the residual circle action described above. Now, it is natural to ask if this spaces are related with the decomposition induced by $G$ on $H(X)$. When $G=\mathbb{T}$ Theorem~\ref{thm:forms} states that they are isomorphic to $H(X)_{k\nu}$; the canonical decomposition of the Hardy space $H(X)$ of the quantitation by the built-in circle action does not play any role. The semi-classical parameter $k$ is the one induced by the ladder $k\,\nu$, $k=0,1,2,\dots$, labeling unitary irreducible representations. 

Further geometrical motivations for this theorem are explained in paper~\cite{Pao2}, where it is proved that $\delta_k$ is an isomorphism for $k$ large enough in the setting when $X$ is the circle bundle of a polarized Hodge manifold whose Grauert tube is $D$. Thus, Theorem~\ref{thm:forms} generalizes the main theorem in~\cite{Pao2} to compact ``quantizable'' pseudo-K\"ahler manifolds. We also refer to~\cite{Pao1} for examples and the explicit expression of the leading term of the asymptotic expansion of $\dim H(X)_{k\nu}$ as $k$ goes to infinity. Along this line of research in \cite{circle} Toeplitz operators were studied for circle action, in \cite{g} the study of asymptotics of compositions of Toeplitz operators with quantomopomorphism is addressed for torus actions. 


\section{Preliminaries}
\subsection{Geometric setting}
\label{sec:Geom}

We recall some notations concerning CR and contact geometry. Let $(X, T^{1,0}X)$ be a compact, connected and orientable CR manifold of dimension $2n+1$, $n\geq 1$, where $T^{1,0}X$ is a CR structure of $X$. There is a unique sub-bundle $HX$ of $TX$ such that $HX\otimes \mathbb{C}=T^{1,0}X \oplus T^{0,1}X$. Let $J:HX\To HX$ be the complex structure map given by $J(u+\ol u)=i u-i\ol u$, for every $u\in T^{1,0}X$. 
By complex linear extension of $J$ to $TX\otimes \mathbb{C}$, the $i$-eigenspace of $J$ is $T^{1,0}X$. We shall also write $(X, HX, J)$ to denote a CR manifold.

Since $X$ is orientable, there always exists a real non-vanishing $1$-form $\omega_0\in\mathcal{C}^{\infty}(X,T^*X)$ so that $\langle\,\omega_0(x)\,,\,u\,\rangle=0$, for every $u\in H_xX$, for every $x\in X$; $\omega_0$ is called contact form and it naturally defines a volume form on $X$. For each $x \in X$, we define a quadratic form on $HX$ by
\[{L}_x(U,V) =\frac{1}{2}\mathrm{d}\omega_0(JU, V),\qquad \forall \ U, V \in H_xX.\]
Then, we extend ${L}$ to $HX\otimes \mathbb{C}$ by complex linear extension; for $U, V \in T^{1,0}_xX$,
\[
	{L}_x(U,\overline{V}) = \frac{1}{2}\,\mathrm{d}\omega_0(JU, \overline{V}) = -\frac{1}{2i}\,\mathrm{d}\omega_0(U,\overline{V}).
\]
The Hermitian quadratic form ${L}_x$ on $T^{1,0}_xX$ is called Levi form at $x$. In the case when $X$ is the circle bundle of an Hodge manifold $(M,\omega)$, the positivity of $\omega$ implies that the number of negative eigenvalues of the Levi form is equal to $n$. The Reeb vector field $R\in\mathcal{C}^\infty(X,TX)$ is defined to be the non-vanishing vector field determined by 
\[	\omega_0(R)\equiv 1,\quad 		\mathrm{d}\omega_0(R,\cdot)\equiv0\ \ \mbox{on $TX$}. \]

Fix a smooth Hermitian metric $\langle\, \cdot \,|\, \cdot \,\rangle$ on $\mathbb{C}TX$ so that $T^{1,0}X$ is orthogonal to $T^{0,1}X$, $\langle\, u \,|\, v \,\rangle$ is real if $u, v$ are real tangent vectors, $\langle\,R\,|\,R\,\rangle=1$ and $R$ is orthogonal to $T^{1,0}X\oplus T^{0,1}X$. For $u \in \mathbb{C}TX$, we write $|u|^2 := \langle\, u\, |\, u\, \rangle$. Denote by $T^{*1,0}X$ and $T^{*0,1}X$ the dual bundles of $T^{1,0}X$ and $T^{0,1}X$, respectively. They can be identified with sub-bundles of the complexified cotangent bundle $\mathbb{C}T^*X$.

Assume that $X$ admits an action of $t$-dimensional torus $\mathbb{T}$. In this work, we assume that the $\mathbb{T}$-action preserves $\omega_0$ and $J$; that is, $t^\ast\omega_0=\omega_0$ on $X$ and $t_\ast J=Jt_\ast$ on $HX$. Let $\mathfrak{t}$ denote the Lie algebra of $T$, we identify $\mathfrak{t}$ with its dual $\mathfrak{t}^*$ by means of the scalar product $\langle\cdot,\cdot\rangle$. For any $\xi \in \mathfrak{t}$, we write $\xi_X$ to denote the vector field on $X$ induced by $\xi$. The moment map associated to the form $\omega_0$ is the map $\mu:X \to \mathfrak{t}^*$ such that, for all $x \in X$ and $\xi \in \mathfrak{t}$, we have 
\begin{equation}\label{E:cmpm}
		\langle \mu(x), \xi \rangle = \omega_0(\xi_X(x)).
\end{equation}

Fix a maximal weight $\nu\neq 0$ in the lattice inside $\mathfrak{t}^*$. Suppose that $i\mathbb{R}_+\cdot \nu$ is transversal to $\mu$, so we have
\begin{equation}\label{eq:trans} \mathfrak{t}^*= i\mathbb{R}_+\cdot \nu \oplus  \mathrm{d}_p\mu(T_pX) 
\end{equation}
and $X_{\nu}$ is a sub-manifold of $X$ of codimension $2n+2-t$. We claim that the action of $\mathbb{T}^{t-1}_{\nu}$ on $X_{\nu}$ is locally free. In fact, by the contrary suppose that there exists $\xi \in \ker{\nu}$ such that $\xi_X(x)=0$ on $T_xX_\nu$, $x\in X_{\nu}$. For each $v\in T_xX$, we have
\[(\mathrm{d}_x\mu(v))(\xi)=\mathrm{d}_x\omega_0(\xi_X,v)=0 \]
which contradicts \eqref{eq:trans}.  

The action of $\mathbb{T}$ restricts to an action of $\mathbb{T}^{t-1}$ whose moment map is given by 
\[\mu_{\vert\mathbb{T}^{t-1}}= p_{\nu}\circ \mu  \,,\quad \text{ where }  p_{\nu}\,:\, \mathfrak{t}^* \rightarrow  \mathfrak{t}_{\nu}^*  \]
is the canonical projection onto the $t-1$-dimensional subspace $\nu^{\perp}$ in $\mathfrak{t}^*$. The transversality condition in \ref{as:2} implies that $0\in \mathfrak{t}_{\nu}^{t-1}$ is a regular value for the moment $\mu_{\vert\mathbb{T}^{t-1}}$. Thus, we have
\[X_{\nu}/\mathbb{T}^{t-1}_{\nu}=\mu_{\vert\mathbb{T}^{t-1}}^{-1}(0)/\mathbb{T}^{t-1}_{\nu}\]
and by assumptions \ref{as:2}, \ref{as:3} and \cite[Section 2.5]{hsiaohuang} (we shall also refer to \cite{gh} for definitions concerning CR structures on orbifolds) we have 
\begin{lem} \label{lem:2.1}
	The space of orbits $X_{\nu}/\mathbb{T}^{t-1}_{\nu}$ is a CR orbifold. Let us denote by $\pi : X_{\nu} \rightarrow X^{\mathrm{red}}_{\nu}$ and $\iota : X_{\nu} \hookrightarrow X$ the natural projection and inclusion, respectively, then there is a unique induced contact form $\omega_0^{\mathrm{red}}$ on $X_{\nu}/\mathbb{T}^{t-1}_{\nu}$ such that
	\[\pi^{*}\omega_0^{\mathrm{red}}=\iota^*\omega_0\,.  \]
	In particular, set $HX_{\nu}= TX_{\nu}\cap HX$, we have
	\[HX= HX_{\nu}\oplus J\,i \nu^{\perp}_X \quad \text{ and }\quad HX_{\nu} = i \nu^{\perp}_X\oplus \mathrm{d}\pi^{*}HX^{\mathrm{red}}_{\nu}\,.  \]
\end{lem}

We will also assume that $\overline{\mathbb{T}^1_{\nu}}$-action is transversal CR, that is, the infinitesimal vector field
$$(\nu_X u)(x)=\frac{\partial}{\partial t}\left(u(\exp(i t\,\nu)\circ x)\right)|_{t=0}\,,\quad \text{for any }u\in C^\infty(X)\,,$$
preserves the CR structure $T^{1,0}X$, so that $\nu_X$ and $T^{1,0}X\oplus  T^{0,1}X$ generate the complex tangent bundle to $X$, $$\mathbb{C}T_xX=\mathbb{C}\nu_X(x)\oplus \mathbb{C}T_x^{1,0}X\oplus \mathbb{C}T_x^{0,1}X \qquad(x\in X)\,.$$ 

We define local coordinates that will be useful later. Recall that $X$ admits a CR and transversal $\overline{\mathbb{T}^1_{\nu}}$-action which is locally free on $X_{\nu}$, $T\in \mathcal{C}^{\infty}(X,\,TX)$ denotes the global real vector field given by this infinitesimal circle action. We will take $T$ to be our Reeb vector field $R$. In a similar way as in Theorem $3.6$ in \cite{hsiaohuang}, there exist local coordinates $v=(v_1,\dots,v_{t-1})$ of $\mathbb{T}^{t-1}_{\nu}$ in a small neighborhood $V_0$ of the identity $e$ with $v(e)=(0,\,\dots,\,0)$, local coordinates $x=(x_1\,\dots,x_{2n+1})$ defined in a neighborhood $U_1\times U_2$ of $p\in X_{\nu}$, where $U_1\subseteq \mathbb{R}^{t-1}$ (resp. $U_{2}\subseteq \mathbb{R}^{2n+2-t}$) is an open set of $0\in \mathbb{R}^{t-1}$ (resp. $0\in \mathbb{R}^{2n+2-t}$) and $p\equiv 0\in \mathbb{R}^{2n+1}$, and a smooth function $\gamma=(\gamma_1,\dots,\,\gamma_{t-1})\in \mathcal{C}^{\infty}(U_2,U_1)$ with $\gamma(0)=0$ such that
\begin{align*} &(v_1,\dots,v_{t-1})\circ (\gamma(x_{t},\dots,x_{2n+1}),x_{t},\dots,x_{2n+1}) \\
	&=(v_1+\gamma_1(x_{t},\dots,x_{2n+1}),\dots,\,v_{t-1}+\gamma_d(x_{t},\dots,x_{2n+1}),\,x_{t},\dots,x_{2n+1}) \end{align*}
for each $(v_1,\dots,\,v_{t-1})\in V_0$ and $(x_{t},\dots,x_{2n+1})\in U_2$. Furthermore, we have
\[\mathfrak{t}=\mathrm{span}\left\{\partial_{{x}_{j}}\right\}_{j=1,\dots, t-1}\,,\quad \mu^{-1}(i\,{\mathbb{R}}_{\nu}\cdot \nu)\cap U =\{x_{2d-t+1} = \dots=x_{2d}=0 \}\,, \]
on $\mu^{-1}(i\,{\mathbb{R}}_{\nu}\cdot \nu)\cap U$ there exist smooth functions $a_j$'s with $a_j(0)=0$ for every $0\leq j\leq t-1$ and independent on $x_1,\dots,x_{2(t-1)},\,x_{2n+1}$ such that
\[J\left(\partial_{{x}_{j}}\right)=\partial_{{x}_{t-1+j}}+a_j(x)\partial_{{x}_{2n+1}}\qquad j=1,\dots,t-1\,, \]
the Levi form ${L}_p$, the Hermitian metric $\langle\,\cdot\,|\,\cdot\,\rangle$ and the $1$-form $\omega_0$ can be written 
\[{L}_p(Z_j,\,\overline{Z}_k)=\mu_j\,\delta_{j,k},\qquad \langle Z_j\vert\,\overline{Z}_k \rangle=\delta_{j,k}\qquad (1\leq j,k\leq n)\,, \]
and
\begin{align*}
	\omega_0(x)=&(1+O(\lvert x\rvert))\mathrm{d}x_{2n+1}+\sum_{j=1}^{t-1} 4\mu_jx_{t-1+j}\mathrm{d}x_j +\sum_{j=t}^n 2\mu_jx_{2j}\mathrm{d}x_{2j-1} \\
	&-\sum_{j=t}^n 2\mu_jx_{2j-1}\mathrm{d}x_{2j}+\sum_{j=r}^{2n} b_jx_{2n+1}\mathrm{d}x_j+O(\lvert x\rvert^2) 
\end{align*}
where $b_{r},\dots,b_{2n}\in \mathbb{R}$,
\[T_p^{1,0}X=\mathrm{span}\{Z_1,\dots,Z_n\} \]
and
\begin{align*}Z_j&=\frac{1}{2}\,(\partial_{{x}_{j}}-i\, \partial_{{x}_{t-1+j}})(p)\qquad(j=1,\dots,t-1)\,,\\ Z_j&=\frac{1}{2}\,(\partial_{{x}_{2j-1}}-i\, \partial_{{x}_{2j}})(p)\qquad(j=t,\dots,n)\,.
\end{align*} 

We need to define in local coordinates we just introduced the phase function of the $\mathbb{T}^{t-1}_{\nu}$-invariant Szeg\H{o} kernel $\Phi_-(x,y)\in\mathcal{C}^\infty(U\times U)$ which is independent of $(x_1,\ldots,x_{t-1})$ and $(y_1,\ldots,y_{t-1})$. Hence, we write $\Phi_-(x,y)=\Phi_-((0,x''),(0,y'')):=\Phi_-(x'',y'')$ and $\mathring{x}'':=(x_{t},\ldots,x_{2n})$,  $\mathring{y}'':=(y_{t},\ldots,y_{2n})$. Moreover, there is a constant $c>0$ such that 
\begin{equation}
	{\rm Im\,}\Phi_-(x'',y'')\geq c\left(\lvert{\mathring{x}''}\rvert^2+\lvert{\mathring{y}''}\rvert^2+\lvert{\mathring{x}''-\mathring{y}''}\rvert^2\right),\ \ \text{for all } ((0,x''),(0,y''))\in U\times U.
\end{equation}
Furthermore, 
\begin{align}\label{eq:phase phi-}
		\Phi_-&(x'', y'')=-x_{2n+1}+y_{2n+1}+2i\sum^{t-1}_{j=1}\abs{\mu_j}y^2_{t-1+j}+2i\sum^{t-1}_{j=1}\abs{\mu_j}x^2_{t-1+j} +i\sum^{n}_{j=t}\abs{\mu_j}\abs{z_j-w_j}^2\\
		& +\sum^{n}_{j=t}i\mu_j(\ol z_jw_j-z_j\ol w_j)+\sum^d_{j=1}(-b_{t-1+j}x_{d+j}x_{2n+1}+b_{t-1+j}y_{t-1+j}y_{2n+1})\notag \\
		&+\sum^n_{j=t}\frac{1}{2}(b_{2j-1}-ib_{2j})(-z_jx_{2n+1}+w_jy_{2n+1}) +\sum^n_{j=t}\frac{1}{2}(b_{2j-1}+ib_{2j})(-\ol z_jx_{2n+1}+\ol w_jy_{2n+1}) \notag \\
		&+(x_{2n+1}-y_{2n+1})f(x, y) +O(\abs{(x, y)}^3)\notag,
\end{align}
where $z_j=x_{2j-1}+ix_{2j}$, $w_j=y_{2j-1}+iy_{2j}$, $j=t,\ldots,n$, $\mu_j$, $j=1,\ldots,n$, $f$ is smooth and satisfies $f(0,0)=0$, $f(x, y)=\ol f(y, x)$. 

We now consider $\overline{\mathbb{T}^1_{\nu}}$ circle action on $X$. Let $p\in\mu^{-1}(i\mathbb{R}_+\cdot \nu)$, there exist local coordinates $v=(v_1,\dots,v_{t-1})$ of $\mathbb{T}^{t-1}$ in a small neighborhood $V_0$ of $e$ with $v(e)=(0,\,\dots,\,0)$, local coordinates $x=(x_1\,\dots,x_{2n+1})$ defined in a neighborhood $U_1\times U_2$ of $p$, where $U_1\subseteq \mathbb{R}^{t-1}$ (resp. $U_{t-1}\subseteq \mathbb{R}^{2n+t}$) is an open set of $0\in \mathbb{R}^{t-1}$ (resp. $0\in \mathbb{R}^{2n+t}$) and $p\equiv 0\in \mathbb{R}^{2n+1}$, and a smooth function $\gamma=(\gamma_1,\dots,\,\gamma_t)\in\mathcal{C}^{\infty}(U_2,U_1)$ with $\gamma(0)=0$ such that $T=-\frac{\pr}{\pr x_{2n+1}}$ and all the properties for the local coordinates defined before hold.
The phase function $\Psi$ satisfies $\Psi(x,y)=-x_{2n+1}+y_{2n+1}+\hat\Psi(\mathring{x}'',\mathring{y}'')$, where $\hat\Psi(\mathring{x}'',\mathring{y}'')\in\mathcal{C}^\infty(U\times U)$ and $\Psi$ satisfies \eqref{eq:phase phi-}. 

\subsection{Hardy spaces}
\label{sec:Hardy}
We denote by $L^2_{(0,q)}(X)$, $q=0,1,\ldots,n$, the completion of $\Omega^{0,q}(X)$ with respect to $(\,\cdot\,|\,\cdot\,)$. We extend $(\,\cdot\,|\,\cdot\,)$ to $L^2_{(0,q)}(X)$ in the standard way. We extend
$\ddbar_{b}$ to $L^2_{(0,q)}(X)$ by
\[
\ddbar_{b}:{\rm Dom\,}\ddbar_{b}\subset L^2_{(0,q)}(X)\To L^2_{(0,q+1)}(X)\,,
\]
where ${\rm Dom\,}\ddbar_{b}:=\{u\in L^2_{(0,q)}(X);\, \ddbar_{b}u\in L^2_{(0,q+1)}(X)\}$ and, for any $u\in L^2_{(0,q)}(X)$, $\ddbar_{b} u$ is defined in the sense of distributions.
We also write
\[
\ol{\pr}^{*}_{b}:{\rm Dom\,}\ol{\pr}^{*}_{b}\subset L^2_{(0,q+1)}(X)\To L^2_{(0,q)}(X)
\]
to denote the Hilbert space adjoint of $\ddbar_{b}$ in the $L^2$ space with respect to $(\,\cdot\,|\,\cdot\, )$. There is a well-defined orthogonal projection
\begin{equation}
	S^{(q)}:L^2_{(0,q)}(X)\To{\rm Ker\,}\Box^q_b
\end{equation}
with respect to the $L^2$ inner product $(\,\cdot\,|\,\cdot\,)$ and let
\[
S^{(q)}(x,y)\in\mathcal{D}'(X\times X,T^{*0,q}X\boxtimes(T^{*0,q}X)^*)
\]
denote the distribution kernel of $S^{(q)}$. We will write $H^q(X)$ for $\ker\Box^q_b$ and for functions we simply write $H(X)$ to mean $H^0(X)$. The distributional kernel $S^{(q)}(x,y)$ was studied in \cite{hsiao}, before recalling an explicit description describing the oscillatory integral defining the distribution $S^{(q)}(x,y)$ we need to fix some notation. To begin, let us review the concept of the Hörmander symbol space. Let $D\subset X$ be a local coordinate patch with local coordinates $x=(x_1,\ldots,x_{2n+1})$. 
\begin{defn}
	For every $m\in\Real$, we denote with  
	\begin{equation} \label{eq:symbol} 
		S^m_{1,0}(D\times D\times\mathbb{R}_+,T^{*0,q}X\boxtimes(T^{*0,q}X)^*)\subseteq \mathcal{C}^\infty(D\times D\times\mathbb{R}_+,T^{*0,q}X\boxtimes(T^{*0,q}X)^*)
	\end{equation} 
	the space of all $a$ 
	such that, for all compact $K\Subset D\times D$ and all $\alpha, \beta\in\mathbb N^{2n+1}_0$, $\gamma\in\mathbb N_0$, 
	there is a constant $C_{\alpha,\beta,\gamma}>0$ such that 
	\[\abs{\pr^\alpha_x\pr^\beta_y\pr^\gamma_t a(x,y,t)}\leq C_{\alpha,\beta,\gamma}(1+\abs{t})^{m-\gamma},\ \ 
	\mbox{for every $(x,y,t)\in K\times\Real_+, t\geq1$}.\]
	For simplicity we denote with $S^m_{1,0}$ the spaces defined in \eqref{eq:symbol}; furthermore, we write 
	\[
	S^{-\infty}(D\times D\times\mathbb{R}_+,T^{*0,q}X\boxtimes(T^{*0,q}X)^*) :=\bigcap_{m\in\Real}S^m_{1,0}(D\times D\times\mathbb{R}_+,T^{*0,q}X\boxtimes(T^{*0,q}X)^*).
	\]
	Let $a_j\in S^{m_j}_{1,0}$, 
	$j=0,1,2,\ldots$ with $m_j\To-\infty$, as $j\To\infty$. 
	Then there exists $a\in S^{m_0}_{1,0}$ 
	unique modulo $S^{-\infty}$, such that 
	$$a-\sum^{k-1}_{j=0}a_j\in S^{m_k}_{1,0}(D\times D\times\mathbb{R}_+,T^{*0,q}X\boxtimes(T^{*0,q}X)^*\big)$$
	for $k=0,1,2,\ldots$. If $a$ and $a_j$ have the properties above, we write $a\sim\sum^{\infty}_{j=0}a_j$ in 
	$S^{m_0}_{1,0}$.   
\end{defn} 

It is known that the characteristic set of $\Box^q_b$ is given by 
\[
	\Sigma=\Sigma^-\cup\Sigma^+,\quad
	\Sigma^-=\set{(x,\lambda\omega_0(x))\in T^*X;\,\lambda<0},
\]
and $\Sigma^+$ is defined similarly for $\lambda>0$. We recall the following theorem (see~\cite[Theorem 1.2]{hsiao}).

\begin{thm}
	Suppose that the Levi form is non-degenerate and $\Box^q_b$ has $L^2$ closed range. Then, there exist continuous operators
	$S_-, S_+: L^2_{(0,q)}(X)\To{\rm Ker\,}\Box^q_b$ such that 
	\[	S^{(q)}=S_-+S_+,\quad	S_+\equiv0\ \ \mbox{if $q\neq n_+$}\]
	and
	\[{\rm WF'\,}(S_-)={\rm diag\,}(\Sigma^-\times\Sigma^-),\quad {\rm WF'\,}(S_+)={\rm diag\,}(\Sigma^+\times\Sigma^+)\ \ \mbox{if $q=n_-=n_+$},
	\]
	where ${\rm WF'\,}(S_-)=\set{(x,\xi,y,\eta)\in T^*X\times T^*X;\,(x,\xi,y,-\eta)\in{\rm WF\,}(S_-)}$, ${\rm WF\,}(S_-)$ is the wave front set of $S_-$ in the sense of H\"ormander. 
	
	Moreover, consider any small local coordinate patch $D\subset X$ with local coordinates $x=(x_1,\ldots,x_{2n+1})$, then 
	$S_-(x,y)$, $S_+(x,y)$ satisfy
	\[
	S_{\mp}(x, y)\equiv\int^{\infty}_{0}e^{i\varphi_{\mp}(x, y)t}s_{\mp}(x, y, t)dt\ \ \mbox{on $D$},
	\]
	with 
	\[s_{\mp}(x,y,t)\sim\sum^\infty_{j=0}s^j_{\mp}(x, y)t^{n-j}\text{ in }S^{n}_{1, 0}(D\times D\times\mathbb{R}_+\,,T^{*0,q}X\boxtimes(T^{*0,q}X)^*)\,,\]
	$s_+(x,y,t)=0$ if $q\neq n_+$ and $s^0_-(x,x)\neq 0$ for all $x\in D$. The phase functions $\varphi_-$, $\varphi_+$  satisfy
	\[\varphi_+, \varphi_-\in\mathcal{C}^\infty(D\times D),\ \ {\rm Im\,}\varphi_{\mp}(x, y)\geq0,\,\varphi_-(x, x)=0,\ \ \varphi_-(x, y)\neq0\ \ \mbox{if}\ \ x\neq y,\]
	and
	\[	d_x\varphi_-(x, y)\big|_{x=y}=-\omega_0(x), \ \ d_y\varphi_-(x, y)\big|_{x=y}=\omega_0(x), \,-\ol\varphi_+(x, y)=\varphi_-(x,y)\,. 	\]
\end{thm}

\begin{rem}
	Kohn~\cite{Koh86} proved that if $q=n_-=n_+$ or $\abs{n_--n_+}>1$ then $\Box^q_b$ has $L^2$ closed range. For a description of the phase function in local coordinates see chapter 8 of part I in \cite{hsiao}.
\end{rem} 

Now we focus on the decomposition induced by the group action on $H^q(X)$. Fix $t\in \mathbb{T}$ and let $t^*:\Lambda^r_x(\Complex T^*X)\To\Lambda^r_{t^{-1}\circ x}(\Complex T^*X)$ be the pull-back map. Since $\mathbb{T}$ preserves $J$, we have $t^*:T^{*0,q}_xX\To T^{*0,q}_{g^{-1}\circ x}X$, for all $x\in X.$ Thus, for $u\in\Omega^{0,q}(X)$, we have $t^*u\in\Omega^{0,q}(X)$. Put $$\Omega^{0,q}(X)_{k\nu }:=\set{u\in\Omega^{0,q}(X);\, (e^{i\theta})^*u=e^{i\,k\langle \nu,\,\theta\rangle}u,\ \ \forall \theta\in \mathbb{R}^r}.$$
Since the Hermitian metric $\langle\,\cdot\,|\,\cdot\,\rangle$ on $\Complex TX$ is $\mathbb{T}$-invariant, the $L^2$ inner product $(\,\cdot\,|\,\cdot\,)$ on $\Omega^{0,q}(X)$ 
induced by $\langle\,\cdot\,|\,\cdot\,\rangle$ is $\mathbb{T}$-invariant. Let $u\in L^2_{(0,q)}(X)$ and $t\in \mathbb{T}$, we can also define $t^*u$ in the standard way. We introduce the following notation 
\[L^2_{(0,q)}(X)_{k\nu}:=\set{u\in L^2_{(0,q)}(X);\, (e^{i\theta})^*u=e^{i\,k\langle \nu,\,\theta\rangle}u,\ \ \forall \theta\in \mathbb{R}^r},\]
and put 
\[({\rm Ker\,}\Box^q_b)_{k\nu}:={\rm Ker\,}\Box^q_b\cap L^2_{(0,q)}(X)_{k\nu}\,.\] 
The equivariant Szeg\H{o} projection is the orthogonal projection 
\[S^{(q)}_{k\nu}:L^2_{(0,q)}(X)\To ({\rm Ker\,}\Box^q_b)_{k\nu}\]
with respect to $(\,\cdot\,|\,\cdot\,)$. Let $S^{(q)}_{k\nu}(x,y)\in\mathcal{D}'(X\times X,T^{*0,q}X\boxtimes(T^{*0,q}X)^*)$ be the distribution kernel of $S^{(q)}_{k\nu}$. The asymptotic expansion for the distributional kernel of the projector $S^{(0)}_{k\nu}$ was studied in \cite{Pao1} when $X$ is the quantizing circle bundle of a given Hodge manifold in Heisenberg local coordinates. Using similar ideas as in \cite{hsiaohuang} one can generalize the results for $(0,q)$-forms in the following theorem, we give a sketch of the proof, which is similar to the proof of \cite[Theorem 1.8]{hsiaohuang}, here the group $G$ in \cite{hsiaohuang} is $\mathbb{T}^{t-1}_{\nu}$ and the circle action in \cite{hsiaohuang} is given by the action of $\overline{\mathbb{T}^1_{\nu}}$.

Since the action is locally free we need to recall some notations from \cite{gh0}. Since the action is locally free $\mu^{-1}(0)/\overline{\mathbb{T}^1_{\nu}}$ is an orbifold, let us denote with $\pi\,:\,\mu^{-1}(0)\rightarrow \mu^{-1}(0)/\overline{\mathbb{T}^1_{\nu}}$ the projection. Furthermore the action of $\overline{\mathbb{T}^1_{\nu}}$ commutes with the one of $\mathbb{T}^{t-1}_{\nu}$, then we have a smooth locally free action of $\mathbb{T}^{t-1}_{\nu}$ on $\mu^{-1}(0)/\overline{\mathbb{T}^1_{\nu}}$. Given $y\in X$ and $g$ in the stabilizer $(\mathbb{T}^{t-1}_{\nu})_{\pi(y)}$, there exist $\lvert (\overline{\mathbb{T}^1_{\nu}})_y^1 \rvert$ elements $e^{ i \,\theta_{g,j}}\in \overline{\mathbb{T}^1_{\nu}}$ ($j=1,\dots,\lvert (\overline{\mathbb{T}^1_{\nu}})_y^1 \rvert$) such that 
\[g\circ y= e^{- i \,\theta_{g,j}}\circ y\,. \]
Thus, all the elements $(e^{ i \,\theta},\,g)\in \overline{\mathbb{T}^1_{\nu}}\times \mathbb{T}^{t-1}_{\nu} =\mathbb{T}$ satisfying 
\[	e^{ i \,\theta}\cdot g\circ y=y\,. \]
are of the form
$(e^{ i  \theta_{g,j}}, g)$ for each $g\in (\mathbb{T}^{t-1}_{\nu})_{\pi(y)}$.

\begin{thm} \label{thm:skn}
	Suppose that $\Box^q_b$ has $L^2$ closed range. Then, there exist continuous operators
	$S^-_{k\nu}, S^+_{k\nu}: L^2_{(0,q)}(X)\To({\rm Ker\,}\Box^q_b)_{k\nu}$ such that 
	\[	S^{(q)}_{k\nu}=S^-_{k\nu}+S^+_{k\nu},\quad	S_+\equiv0\ \ \mbox{if $q\neq n_+$}\]
	Suppose for simplicity that $q\neq n_+$. If $q\neq n_-$, then $S^{(q)}_{k\nu}\equiv O(k^{-\infty})$ on $X$.
	
	Suppose $q= n_-$ and let $D$ be an open set in $X$ such that the intersection $\mu^{-1}(i\,\mathbb{R}_+\cdot \nu)\cap D= \emptyset$. Then $S^{(q)}_{k\nu}\equiv O(k^{-\infty})$ on $D$.
	
	Let $p\in \mu^{-1}(i\,\mathbb{R}_+\cdot \nu)$ and let $U$ a local neighborhood of $p$ with local coordinates $(x_1,\dots\,x_{2n+1})$. Then, if $q= n_-$, for every fix $y\in U$, we consider $S^{(q)}_{k\nu}(x,y)$ as a $k$-dependent smooth function in $x$, then 
	\[S^{(q)}_{k\nu}(x,\,y)= \sum_{h\in G_{\pi(y)}} \sum_{j=1}^{\lvert S^1_x\rvert} e^{ik\, \theta_{h,j}} e^{i k\lVert\nu\rVert\,\Psi(x,\,y)}\,b(x,\,y,\,k\lVert\nu\rVert)+O(k^{-\infty})\,. \]
	for every $x\in U_y$, where $U_y$ is a small open neighborhood of $y$. The phase function $\Psi(x,\,y)$ is defined in local coordinates in the end of Section \ref{sec:Geom}, the symbol satisfies
	\[b(x,\,y,\,k\lVert\nu\rVert)\in S^{n+(1-t)/2}_{\mathrm{loc}}(1,\,U\times U,\,T^{*\,(0,q)}X\boxtimes(T^{*\,(0,q)})^*)\, \]
	and the leading term of $b$ is non-zero.
\end{thm}
\begin{proof}
	Suppose $q=n_-$, on small local neighborhood $D$ of a point $p\in X_\nu$ we have
	\[ S_{k\nu}^{(q)}(x,\,y)=\frac{1}{(2\,\pi)^t}\int_{-\pi}^{\pi}\dots \int_{-\pi}^{\pi} e^{i k\,\langle \nu,\,\theta \rangle}\,S^{(q)}( x,\,e^{i \,\theta}\cdot y) \,\mathrm{d}\theta\, \]
	where $\theta=(\theta_1,\dots,\theta_t)$.
	It is easy to prove that the oscillatory integral has a rapidly decreasing asymptotic as $k\rightarrow +\infty$ far away from a local neighborhood of those elements $(e^{ i \,\theta},\,g)\in \overline{\mathbb{T}^1_{\nu}}\times \mathbb{T}^{t-1}_{\nu} =\mathbb{T}$ such that
	\[	e^{ i \,\theta}\cdot g\circ y=y\,. \]
	
	 We shall consider the case of a local neighborhood of the identity in $\mathbb{T}$, we set $\theta_t$ the variable for circle action of $\overline{\mathbb{T}^1_{\nu}}$ and $(\theta_1,\dots,\theta_{t-1})$ the variables for $\mathbb{T}^{t-1}_{\nu}$. We can then use local coordinates defined in \ref{sec:Geom}, we have
	\begin{align} \label{eq:sknu}
	S_{k\nu}^{(q)}(x,\,y)=\frac{1}{2\,\pi}\int_{-\pi}^{\pi} e^{i k\, \lVert\nu\rVert\theta_t -i\,kx_{2n+1}+i\,ky_{2n+1}}\,S_{\mathbb{T}^{t-1}}^{(q)}(\mathring{x},\,e^{i \,\theta_t}\cdot \mathring{y}) \,\mathrm{d}\theta_t 
	\end{align}
	where
	\begin{equation} \label{eq:st}
		S_{\mathbb{T}^{t-1}}^{(q)}({x},\,{y})=\frac{1}{(2\,\pi)^{t-1}}\int_{-\pi}^{\pi}\dots \int_{-\pi}^{\pi} \,S^{(q)}( x,\,e^{i \,\mathrm{diag}(\theta_1,\dots,\theta_{t-1})}\cdot y) \,\mathrm{d}\theta_1\,\dots \mathrm{d}\theta_{t-1}\,, 
	\end{equation}
	and $\mathring{x}=(x_1,\dots,x_n,0)$ and $\mathring{y}=(y_1,\dots,y_n,0)$. The action of $\mathbb{T}$ restricts to an action of $\mathbb{T}^{t-1}$ whose moment map is given by 
	\[\mu_{\vert\mathbb{T}^{t-1}}= p_{\nu}\circ \mu  \,,\quad \text{ where }  p_{\nu}\,:\, \mathfrak{t}^* \rightarrow  \mathfrak{t}_{\nu}^*  \]
	is the canonical projection onto the $t-1$-dimensional subspace $\nu^{\perp}$ in $\mathfrak{t}^*$. Since we are assuming $X_{\nu} \neq \emptyset$ we have that $0\in \mathfrak{t}_{t-1}^*$ does lie in the image of the moment map $\mu_{\vert\mathbb{T}^{t-1}}$. The proof follows in a similar way as in Theorem $1.8$ in \cite{hsiaohuang} where here we have the action of $G=\mathbb{T}^{t-1}$ with moment map $\mu_{\mathbb{T}^{t-1}}$ and
	the $k\nu$ Fourier components are the ones induced by the transversal $\overline{\mathbb{T}^1_{\nu}}$-action; notice that \cite[Assumption 1.7]{hsiaohuang} is satisfied. We shall also refer to \cite[Theorem 5.4]{gh0} for the locally free action case.
\end{proof}

For ease of notation we introduce a definition to say that an operator behaves micro-locally as equivariant Szeg\H{o} projector we just studied. For simplicity we assume $q=n_-$.

\begin{defn}[{Equivariant Szeg\H{o} type operator}]\label{def:szegotype}
	Suppose that $q=n_-$ and consider $H: \Omega^{0,q}(X)\To\Omega^{0,q}(X)$ be a continuous operator with distribution kernel $$H(x,y)\in\mathcal{D}'(X\times X,T^{*0,q}X\boxtimes(T^{*0,q}X)^*)\,.$$
	We say that $H$ is a complex Fourier integral operator of equivariant Szeg\H{o} type of order $n+(1-t)/2 \in\mathbb Z$ if $H$ is smoothing away $\mu^{-1}(i\,\mathbb{R}_+\cdot \nu)$ and for given $p\in \mu^{-1}(i\,\mathbb{R}_+\cdot \nu)$ let $D$ a local neighborhood of $p$ with local coordinates $(x_1,\dots\,x_{2n+1})$. Then, the distributional kernel of $H$ satisfies
	\[H_{k}(x,y)=  e^{i k\,\Psi(x,\,y)}\,a(x,\,y,\,k)+O(k^{-\infty})\quad \text{ on }D\]
	where $a\in S^{n+(1-t)/2}_{1,0}(D\times D\times\mathbb{R}_+,T^{*0,q}X\boxtimes(T^{*0,q}X)^*)$, and $\Psi$ is as in the end of Section \ref{sec:Geom}.
\end{defn}

For $k=0$, $H^{(q)}(X)_0$ is the space of $\mathbb{T}$-fixed vectors, by Theorem $1.5$ in \cite{hsiaohuang} we have $\dim H_0^{(q)}(X)< +\infty$ since when $0\notin \mu(X)$ the projector $S^{(q)}_{\mathbb{T}}$ is smoothing. Now, given $k_1, k_2\in \mathbb{Z}$ with $k_1, k_2\neq 0$, consider $f_1\in H^{(q)}(X)_{k_1\nu}$ and $f_2\in H^{(q)}(X)_{k_2\nu}$, we have $f_1\cdot f_2 \in H^{(q)}(X)_{(k_1+k_2)\nu}$; we can study $\dim H(X)_{k\nu}$ for $k$ large, we have 
\[\dim H^{(q)}(X)_{k\nu} = \int_X S_{k\nu}^{(q)}(x,\,x)\,\mathrm{dV}_X(x)\,. \]
These dimensions can be studied as $k\rightarrow +\infty$ by using the microlocal properties of the Szeg\H{o} kernel and Stationary Phase Lemma in a similar way as in Corollary $1.3$ in \cite{Pao1} we have $\dim H^{(q)}(X)_{k\nu}= O(k^{d+1-t})$. So, we have the following generalization of Theorem $A.3$ in \cite{gs}, where it is proved for spaces of functions $H(X)_{k\nu}$.

\begin{lem}
	If $0\notin \mu(X)$, then $H^q(X)_{k\nu}$ are finite dimensional.
\end{lem}

\section{Proof of Theorem \ref{thm:forms}}
\label{sec:forms}

In this section we shall explain how to prove asymptotic commutativity for quantization and reduction for spaces of $(0,q)$-forms. We recall that $X^{\mathrm{red}}_{\nu}$ is a CR orbifold whose Levi form is non-degenerate, with $n_--r$ negative eigenvalues; let us denote with  $S^{(q)}_{\mathrm{red}}$ the corresponding Szeg\H{o} kernel for $(0,q)$-forms whose $k$-th Fourier components are the one induced by the $\overline{\mathbb{T}^1_{\nu}}$-action on $X^{\mathrm{red}}_{\nu}$. We shall recall briefly its microlocal expression by \cite[Theorem 1.2]{gh}. Now, let us denote by $e^{i\theta}\cdot$ the transversal and CR locally free $\overline{\mathbb{T}^1_{\nu}}$-action and we take the Reeb vector field $R$ to be the vector field on $X$ induced by it. 

Let $q=n_--r$, and consider an open set $U\subset X$, $p\in U$, and an orbifold chart $(\widetilde{U},G_U)\rightarrow U$, we denote by $\widetilde{x}$ the coordinates on $\widetilde{U}$. For every $\ell\in\mathbb N$, put 
\[X_\ell:=\{x\in X;\, e^{i\theta}x\neq x, \theta\in [0,{2\pi}/{\ell}[, e^{i\,{2\pi}/{\ell}}x=x\}.\]
With the assumptions and notations used above, assume  that $p\in X_\ell$, for some $\ell\in\mathbb N$. We have as $k\To+\infty$,
\begin{equation}\label{eq:szegos1free}
	S^{(q)}_{\mathrm{red},k}(x,y)=\sum^{\ell-1}_{j=0}\sum_{g\in G_U} e^{\frac{2\pi kj}{\ell}}\, e^{ik\, \Psi(\widetilde{x}, e^{i\frac{2\pi j}{\ell}}\cdot g\cdot\widetilde{y})}b(\widetilde{x},\,e^{i\frac{2\pi j}{\ell}}\cdot g\cdot \widetilde{y},k)+O(k^{-\infty})
\end{equation}
where the phase function 
\begin{align*}
	\begin{split}
		&\Psi\in\,{C}^\infty(\widetilde U\times\widetilde U)\,,\qquad \Psi(\widetilde x,\widetilde x)=0,\ \ \mbox{for all $\widetilde x\in\widetilde U$},\\
		&\inf_{e^{i\theta}\in S^1}\set{\mathrm{dist}^2(\widetilde x,e^{i\theta}\,\widetilde y)}/C\leq {\rm Im\,}\Psi(\widetilde x,\widetilde y)\leq C\inf_{e^{i\theta}\in S^1}\set{\mathrm{dist}^2(\widetilde x,e^{i\theta}\,\widetilde y)}\,
	\end{split}
\end{align*}
for each $(\widetilde x,\widetilde y)\in\widetilde U\times\widetilde U$, $C>1$ is a constant; and the symbol satisfies 
\begin{equation*}
	b(\widetilde x,\widetilde y,k)\sim\sum^{+\infty}_{j=0}b_j(\widetilde x,\widetilde y)\,k^{n-j} \quad \text{in} \quad
			S^{n-(t-1)}(1;\widetilde U\times\widetilde U,T^{*0,q}X\boxtimes(T^{*0,q}X)^*)\,
\end{equation*}
and $b_0(\widetilde x,\widetilde x)$ is non zero.

Let us recall briefly why we have a local chart $(\widetilde{U},\,G_U)\rightarrow U$, for ease of notation let us put $G=\mathbb{T}^{t-1}_{\nu}$. We recall that, for every $x\in X_{\nu}$, by the slice Theorem a neighborhood of any orbit $\mathbb{T}^{t-1}_{\nu}\cdot x=x_0$ is equivariantly diffeomorphic to a neighborhood of the zero section of the associated principal bundle $$G\times_{G_x} N_x\,,$$ where $N_x$ is the normal space to $G\cdot x$ in $X_{\nu}$ and $G_x$ is the stabilizer of $x$ for the action of $G$, which is finite. Therefore, for some $\epsilon>0$ and for an open ball $B_{2{e}}(\epsilon)\subseteq N_x$, one has a homeomorphism $B_{2{e}}(\epsilon)/G_x \cong \overline{U}$ onto some neighborhood of $x_0$ in $X^{\mathrm{red}}_{\nu}$. 

Since $\nu^{\perp}_X:=\mathrm{val}(\nu^{\perp})$ is orthogonal to $$H_xX_{\nu}\cap J_x H_xX_{\nu}\quad \text{ and } \quad H_xX_{\nu}\cap J_x H_xX_{\nu}\subset (\nu^{\perp}_X)^{\perp_b}$$ for every $x\in X_{\nu}$, we can find a $\mathbb{T}$-invariant orthonormal basis $\{Z_1,\,\dots,\, Z_n\}$ of $T^{*\,0,1}X$ on $X_{\nu}$ such that for each $j,\,k=1,\dots n$,
\[L_x(Z_j(x),\,\overline{Z}_j(x))= \delta_{jk}\,\lambda_j(x)\,,  \]
where
\[ Z_j(x)\in (\nu^{\perp}_{X,x}+i J \nu^{\perp}_{X,x})\text{ for each $j=1,\,\dots,\, t-1$ }\,, \]
\[  Z_j(x)\in H_xX_{\nu}\cap J_x(H_xX_{\nu}) \text{ for each $j=r,\,\dots,\, n$ }\,. \]
Let $\{Z_1^*,\,\dots,\, Z_n^*\}$ denote the orthonormal basis of $T^{*\,0,1}X$ on $X_{\nu}$, dual to $Z_1,\,\dots,\, Z_n$. Fix $s = 0,\, 1,\, 2,\dots,\, n-r+1$. For $x \in X_{\nu}$, put
\[B_x^{*\,0,s}X=\left\{\sum_{r\leq j_1<\dots<j_s\leq n} a_{j_1,\dots,j_s}Z_{j_1}^*,\,\dots, Z_{j_s}^*;\,a_{j_1,\dots,j_s}\in\mathbb{C},\, \right\} \]
and let $B^{*\,0,s}X$ be the vector bundle of $X_{\nu}$ with fibre $B_x^{*\,0,s}$, $x\in X_{\nu}$. Let $C^{\infty}(X_{\nu}, B^{*\,0,s}X)^{\mathbb{T}}$ denote the set of all $\mathbb{T}$-invariant sections of $X_{\nu}$ with values in $B_x^{*\,0,s}X$; let
\[\iota_{\mathrm{red}} \,:\,C^{\infty}(X_{\nu}, B^{*\,0,s}X)^{\mathbb{T}} \rightarrow \Omega^{0,s}(X^{\mathrm{red}}_{\nu}) \]
be the natural identification. Before defining the map between $H^q(X)_{k\nu}$ and $H^{q-r}(X^{\mathrm{red}}_{\nu})_{k}$ we need one more piece of notation. We can assume that $\lambda_1 < 0, \cdots, \lambda_{r} < 0$ and also $\lambda_{t}< 0,\dots ,\, \lambda_{n_--r+t-1} < 0$. For $x\in X_{\nu}$, set
\[\hat{\mathcal{N}}(x,n_-)=\{c\,Z_t^*\wedge\,\cdots\,\wedge Z_{n_--r+t-1}^*\,,c\in \mathbb{C} \} \]
and let $\hat{p}\,:\,\mathcal{N}(x,n_-)\rightarrow \hat{\mathcal{N}}(x,n_-)$ to be
\[\hat{p}(c\, Z_1^*\wedge\,\cdots\,\wedge Z_r^*\wedge Z_t^*\wedge\,\cdots\,\wedge Z_{n_--r+t-1}^* ):= c\,Z_t^*\wedge\,\cdots\,\wedge Z_{n_--r+t-1}^*  \]
where $c$ is a complex number. Put $\iota_{\nu} \,:\, X_{\nu} \rightarrow X$ be the natural inclusion and let $\iota_{\nu}^* \,:\, \Omega^{0,q}(X)\rightarrow \Omega^{0,q}(X_{\nu})$ be the pull-back of $\iota_{\nu}$.

Let $q = n_-$; now, inspired by \cite{hsiaohuang}, we define the map 
\[\sigma_{k\nu}\,:\,H^{q}(X)_{k\nu}\rightarrow H^{q-r}(X^{\mathrm{red}}_{\nu})_k \]
given by
\begin{equation} \label{eq:sigma}
	\sigma_{k\nu}(u):=(k\nu)^{(t-1)/4}\,S^{(q-r)}_{\mathrm{red},k}\circ\iota_{\mathrm{red}}\circ \hat{p} \circ \tau_{x,n_-}\circ e \circ \iota^*_{\nu} \circ S^{(q)}_{k\nu}(u) \,, 
\end{equation}
here, $e(x)$ is a $\mathbb{T}$-invariant smooth function on $X_{\nu}$ that can be found explicitly. Notice that operators $S^{(q-r)}_{\mathrm{red},k}$ and $S^{(q)}_{k\nu}$ appearing in \eqref{eq:sigma} are known explicitly. $S^{(q-r)}_{\mathrm{red},k}$ is the $k$-th Fourier component of the standard Szeg\H{o} kernel for the CR manifold $X^{\mathrm{red}}_{\nu}$, on the other hand $S^{(q)}_{k\nu}$ is described in Section \ref{sec:Hardy}. 

Before stating the next theorem we shall specialize the local coordinates defined in Section \ref{sec:Geom}. Let us consider $p\in X_{\nu}$ and let $x=(x_1,\,\dots,\,x_{2n+1})$ be the local coordinates in an open neighborhood $U$ of $p$ defined in Section \ref{sec:Geom}. We may assume that $U=U_1\times U_2 \times U_3\times U_4$, where $U_1\subset \mathbb{R}^{t-1}$, $U_2\subset \mathbb{R}^{t-1}$ are open sets of $0\in \mathbb{R}^{t-1}$, $U_3\subset \mathbb{R}^{2n-2(t-1)}$ is an open set of $0\in \mathbb{R}^{2n-2(t-1)}$ and $U_4$ is an open set of $0\in \mathbb{R}$. From now on, we can identify $U_2$ with
\[\{(0,\,\dots,\,0,\,x_{t},\dots,\,x_{2(t-1)},\,0,\,\dots,\,0)\in U\,:\,(x_t,\,\dots,\,x_{2(t-1)})\in U_2\}\]
and $U_3$ with
\[\{(0,\,\dots,\,0,\,x_{2t-1},\dots,\,x_{2n},\,0)\in U\,:\,(x_t,\,\dots,\,x_{2n})\in U_3\}\,.\]
For a given orbifold chart $(\widetilde{U},\,G_U)\rightarrow U$ we write $\tilde{U}_i$ for the corresponding open sets in $\tilde{U}$. Eventually we recall that $\widetilde{x}'':=(x_{2t-1},\,\dots,\,x_{2n+1})$. 

\begin{thm} \label{thm:sigknu}
	Under the assumptions above, if $x\notin X_{\nu}$, then for every sufficiently small open set $D$ of $x$ with $\overline{D}\cap X_{\nu}=\emptyset$, we have $\sigma_{k\nu}=O(k^{-\infty})$ on $X^{\mathrm{red}}_{\nu}\times D$.
	
	Let $\pi\,:\,X_{\nu} \rightarrow X^{\mathrm{red}}_{\nu}$ the projection. If $x,y\in X_{\nu}$ and $\pi(x)\neq \pi(e^{i\theta}\cdot y)$ for every $e^{i\theta}\in \overline{\mathbb{T}^1_{\nu}}$, then there exist open set $U$ of $\pi(x)$ in $X^{\mathrm{red}}_{\nu}$ and $V$ of $y$ in $X$ such that $\sigma_{k\nu}=O(k^{-\infty})$ on $U\times V$.
	
	Eventually, let $p\in X_{\nu}\cap X_{\ell}$, using the local coordinates defined above, we have 
	\[\sigma_{k\nu}(\tilde{x}'',y'')= \sum^{\ell-1}_{j=0}\sum_{g\in G_U} e^{\frac{2\pi kj}{\ell}} e^{ik\, \Psi(\widetilde{x}, \,e^{\frac{2\pi j}{\ell}}\cdot g\cdot\widetilde{y})} \,\alpha(\widetilde{x},\, e^{\frac{2\pi j}{\ell}}\cdot  g\cdot \widetilde{y},k)+O(k^{-\infty})\quad \text{on }(\tilde{U}_3\times \tilde{U}_4)\times \tilde{U}\]
	where
	\[\alpha(\tilde{x}'',y'',k)\in S^{n-\frac{3}{4}(t-1)}_{\mathrm{loc}}(1;\,(\tilde{U}_3\times \tilde{U}_4)\times \tilde{U},\,T^{*,(0,q-r)}X^{\mathrm{red}}_{\nu}\boxtimes (T^{*,(0,q)}X)^*) \]
	and the leading term $\alpha_0$ in the expansion of $\alpha$ can be computed explicitly along the diagonal.
\end{thm}
\begin{proof}
	Since, by Theorem \ref{thm:skn}, $S^{(q)}_{k\nu}$ has rapidly decreasing asymptotics as $k$ goes to infinity away $X_{\nu}$, we obtain that $\sigma_k=O(k^{-\infty})$ on $X^{\mathrm{red}}_{\nu}\times D$.
	
	Now, let us prove the second statement. If $x,y\in X_{\nu}$ and $\pi(x)\neq \pi(e^{i\theta}\cdot y)$ for every $e^{i\theta}\in \overline{\mathbb{T}^1_{\nu}}$, since $S^{(q-r)}_{\mathrm{red},k}$ is smoothing away from diagonal, then there exist open set $U$ of $\pi(x)$ in $X^{\mathrm{red}}_{\nu}$ and $V$ of $y$ in $X$ such that $\chi\,S^{(q-r)}_{\mathrm{red},k}\eta = O(k^{-\infty})$ on $X^{\mathrm{red}}_{\nu}$,
	where ${\chi}\in C^{\infty}_0(U)$ and $\eta\in C^{\infty}_0(V)$. Furthermore by \eqref{eq:sknu}, we see that, for sufficiently small neighborhoods of $x$ and $y$, we can integrate by parts in $\mathrm{d}\theta_t$ and see that $S^{(q)}_{k\nu}$ has rapidly decreasing asymptotic. The second statements follows. 
	
	Eventually, let $p\in X_{\nu}$ and consider a small open neighborhood $U$ with coordinates defined as above. Let $\chi\in C^{\infty}_0(\tilde{U}_3)$ be a $G_U$ invariant bump function and assume $\chi=1$ on some neighborhood of $p$, it extends naturally to a function on $\mathbb{T}\cdot \tilde{U}_3$.  Let us put $U^{\sharp} = \{\pi(x)\,:\,x\in U\}$. Let  $\eta\in C^{\infty}_0(X^{\mathrm{red}}_{\nu})$ such that $\eta=1$ on some neighborhood of $U^{\sharp}$ and $$\mathrm{supp}(\eta)\subset \{\pi(x)\in X^{\mathrm{red}}_{\nu} \,:\, x\in X_{\nu},\,\chi(x)=1 \}\,.$$ Thus, we get
	\[\eta\,\sigma_{k\nu}\sim (k\nu)^{(t-1)/4}\,\eta\,S^{(q-r)}_{\mathrm{red},k}\circ\iota_{\mathrm{red}}\circ \hat{p} \circ \tau_{x,n_-}\circ e \circ \iota^*_{\nu} \circ \chi \,S^{(q)}_{k\nu} \,.  \]
	Now, we can compose the operators and we can use the complex stationary phase formula of \cite{ms}, we get the theorem.
\end{proof}

In fact, $\sigma_{k\nu}$ is an isomorphism for $k$ large if we can prove that  
\[\sigma_{k\nu}\,:\, H^{q}(X)_{k\nu}\rightarrow H^{q-r}(X^{\mathrm{red}}_{\nu})_k\quad\text{ and }\quad \sigma_{k\nu}^*\,:\, H^{q-r}(X^{\mathrm{red}}_{\nu})_k\rightarrow H^{q}(X)_{k\nu} \] 
are injective for $k$ large. Notice that, by Theorem \ref{thm:sigknu}, we have
\begin{align} \label{eq:sigmaknustar}
	\sigma_{k\nu}^*({x}'',\tilde{y}'')= \sum^{\ell-1}_{j=0}\sum_{g\in G_U} e^{\frac{2\pi kj}{\ell}} e^{-i k\,\overline{\Psi({x}'',e^{\frac{2\pi j}{\ell}}\cdot g\cdot\tilde{y}'')}} \,\beta({x}'',\,e^{\frac{2\pi j}{\ell}}\cdot g\cdot\tilde{y}'',k)+O(k^{-\infty})
\end{align}
on $\tilde{U}\times (\tilde{U}_3\times \tilde{U}_4) $ where we can check $\beta_0(\tilde{x}'',\tilde{x}'')=\alpha_0(\tilde{x}'',\tilde{x}'')$. 

The injectivity of the map $\sigma_{k\nu}$ is a consequence of the following theorem which is an adaptation of proof of the main theorem in \cite{hsiaohuang}.

\begin{thm} There exists a Fourier integral operator $R_k$ of \textit{equivariant Szeg\H{o} type} of degree $n-(t-1)/2$ such that
\begin{equation} \label{eq:final}
	\sigma_{k\nu}^*\sigma_{k\nu} \equiv c_0 \,(1+R_{k})\,S^{(q)}_{k\nu} 
\end{equation}
where $c_0$ is a positive constant and $1+R_{k}\,:\,\Omega^{0,q}(X)\rightarrow \Omega^{0,q}(X)$ is an injective Fourier integral operator of \textit{equivariant Szeg\H{o} type}. 
\end{thm}
\begin{proof} By Theorem \ref{thm:sigknu} and equation \eqref{eq:sigmaknustar} we can compose $\sigma_{k\nu}^*$ and $\sigma_{k\nu}$ using the complex stationary phase formula of \cite{ms}. We can pick $e$ in the definition of $\sigma_{k\nu}$ so that the leading term of the symbol of $\sigma_{k\nu}^*\sigma_{k\nu}$ agrees with the one of $S^{(q)}_{k\nu}$ along the diagonal. Thus we can find an operator 
	\[R_k(x,\,y)= \sum^{\ell-1}_{j=0}\sum_{g\in G_U} e^{\frac{2\pi kj}{\ell}} e^{ik\,\Psi(x'',e^{\frac{2\pi j}{\ell}}\cdot g\cdot y'')}\,r(x'',\,e^{\frac{2\pi j}{\ell}}\cdot g\cdot y'',k)+O(k^{-\infty})\quad \text{on }\tilde{U}\times \tilde{U} \]
such that 
\[r\in S^{n-(t-1)/2}(1;\,\tilde{U}\times \tilde{U},\,T^{*(0,q)}X\boxtimes (T^{*(0,q)}X)^*) \]
and
\[\lvert r_0(x,\,y) \rvert \leq C\lvert (x,y)-(x_0,x_0) \rvert \]
for all $x_0\in X_{\nu}\cap \tilde{U}$, where $C>0$ is a constant.

The injectivity of the operator $1+R_{k}$ follows from direct computations of the leading symbol of $R_k$ which vanishes at $\mathrm{diag}(X_{\nu} \times X_{\nu} )$. In fact, as a consequence of this, by Lemma $6.7$ and Lemma $6.8$ in \cite{hsiaohuang} we have that $\lVert R_k u \lVert \leq \epsilon_k \lVert u\rVert$ for all $u \in \Omega^{0,q}(X)$, for all $k\in \mathbb{N}$ where $\epsilon_k$ is a sequence with $\lim_{k\rightarrow +\infty } \epsilon_k = 0$; this in turn implies, if $k$ is large enough, that the map $1+R_{k}$ is injective.
\end{proof}

Analogously, the injectivity of the map $\sigma_{k\nu}^*$ follows by studying $\sigma_{k\nu}\sigma_{k\nu}^*$. We don't repeat the proof here since it is an application of the stationary phase formula for Fourier integral operators of complex type, see \cite{ms}, and it follows in a similar way as above.

\bigskip
\textbf{Acknowledgments:} I express my gratitude to the referees for their valuable suggestions and corrections, which have contributed to the improvement of this paper. Furthermore, I am indebted to Roberto Paoletti for a remark on the previous version of this work.

I would like to acknowledge the support received from the National Center for Theoretical Sciences in Taiwan, where this project was initiated during my postdoctoral fellowship.

Additionally, I extend my thanks to the Mathematics Institute of Universität zu K\"oln for their hospitality throughout my scholarship: the author is supported by INdAM (Istituto Nazionale di Alta Matematica) foreign scholarship.

\end{document}